\documentclass[a4paper,11pt]{amsart}
\usepackage{hyperref}
\usepackage[usenames,dvipsnames,svgnames,table]{xcolor}
\usepackage[all]{xy}
\usepackage{enumitem}
\usepackage[normalem]{ulem}
\newcommand{\tsk}[1]{\textcolor{YellowOrange}}

\usepackage{tikz}
\usepackage{tikz-cd}
\usepackage{booktabs}

\makeatletter
\def\@endtheorem{\endtrivlist}% NEW
\makeatother
\usepackage[active]{srcltx}
\usepackage{amsmath}
\usepackage{amssymb}
\usepackage{amscd}
\usepackage{amsthm}
\usepackage[latin1]{inputenc}
\usepackage{mathrsfs}  

\usepackage{nicefrac}

\newtheorem{teo}{Theorem}[section]
\newtheorem{defin}[teo]{Definition}
\newtheorem{prop}[teo]{Proposition}
\newtheorem{cor}[teo]{Corollary}
\newtheorem{lemma}[teo]{Lemma}

\theoremstyle{definition}

\newtheorem{remark}[teo]{Remark}

\newtheoremstyle{dico}% name of the style to be used
 {\baselineskip}   % ABOVESPACE
  {\topsep}   % BELOWSPACE
  {}  % BODYFONT
  {0pt}       % INDENT (empty value is the same as 0pt)
  {} % HEADFONT
  {.}         % HEADPUNCT
  {5pt plus 1pt minus 1pt} % HEADSPACE
  {}          % CUSTOM-HEAD-SPEC
\theoremstyle{dico}

\numberwithin{equation}{section}

\newcommand{\ra}{\rightarrow}
\newcommand{\C}{\mathbb{C}}

\newcommand{\vacuo}{\emptyset}

\newcommand{\Ad}{\operatorname{Ad}}
\newcommand{\ad}{{\operatorname{ad}}}

\renewcommand{\phi}{\varphi}

\newcommand{\sx}{\langle}
\newcommand{\xs}{\rangle}

\newcommand{\Ga}{\Gamma}

\renewcommand{\sp}{\mathfrak{sp}}

\renewcommand{\phi}             {\varphi}

\newcommand{\sieg}{\mathfrak{S}}

\newcommand{\M}{\mathsf{M}}

\newcommand{\T}{\mathsf{T}}

\newcommand{\A}{\mathsf{A}}
\newcommand{\Z}{\mathsf{Z}}
\newcommand{\Ag}{\mathsf{A}_g}

\newcommand{\Sp}                {\operatorname {Sp}}
    
\newcommand{\tr}              {\operatorname {tr}}

\newcommand{\diag}{\operatorname{diag}}

\newcommand{\mihi}[1]{}

\def\LaTeX{%
	\let\Begin\begin
	\let\salta\relax
	\let\finqui\relax
	\let\futuro\relax}
      \LaTeX

\def\Edef{\end{definition}}
\def\Blem{\Begin{lemma}}
\def\Elem{\end{lemma}}
\def\Bdim{\Begin{proof}}
\def\Edim{\end{proof}}

\def\Eprop{\end{prop}}

\def\Ecor{\end{corollary}}

\def\Ethm{\end{theorem}}

\def\Erem{\end{remark}}

\def\p{\mathfrak{p}}
\def\k{\mathfrak{k}}
\def\u{\mathfrak{u}}
\def\g{\mathfrak{g}}
\def\Z{\mathfrak{Z}}

\def\erre{\mathbb{R}}

\begin{document}

 \author{Carolina Tamborini}
 \title
[Symmetric spaces uniformizing Shimura varieties]{Symmetric spaces uniformizing Shimura varieties in the Torelli locus}

\address{Universit\`{a} di Pavia}
 \email{c.tamborini2@campus.unimib.it}
 %  32G20 %Period matrices, variation of Hodge structure; degenerations
 %  and 14K22 %complex multiplication
% }

\thanks{The author was partially supported by MIUR PRIN 2017
  ``Moduli spaces and Lie Theory'',  by MIUR, Programma Dipartimenti di Eccellenza
   (2018-2022) - Dipartimento di Matematica ``F. Casorati'',
   Universit\`a degli Studi di Pavia and by INdAM (GNSAGA)}

\maketitle

\begin{abstract}
	An algebraic subvariety $Z$ of $\A_g$ is  \emph{totally geodesic} if it is the image via the natural projection map of some totally geodesic submanifold $X$ of the Siegel space. We say that $X$ is the symmetric space \emph{uniformizing} $Z$.
	In this paper we determine which symmetric space uniformizes each of the low genus counterexamples to the Coleman-Oort conjecture obtained studying Galois covers of curves. It is known that the counterexamples obtained via Galois covers of elliptic curves admit
	two fibrations in totally geodesic subvarieties. The second result of the paper studies the relationship between these fibrations and the uniformizing symmetric space of the examples.
\end{abstract}

\setcounter{tocdepth}{1}

\tableofcontents{}

\section{Introduction}

\subsection {} Let $\M_g$ denote the moduli space of smooth complex algebraic curves of genus $g$, $\Ag$ the moduli space
of principally polarized abelian varieties of dimension $g$ over $\mathbb{C}$ and $j : \M_g \ra \A_g$ the Torelli map. We call $\T_g=\overline{j(\M_g)}$ the \emph{Torelli locus.} The moduli space $\A_g=\sieg_g/Sp(2g, \mathbb{Z})$ is a quotient of the Siegel space $\sieg_g$, which is an irreducible hemitian symmetric space of the non-compact type. In particular, $\sieg_g$ induces on $\A_g$ a locally symmetric metric, called the \emph{Siegel metric.} Denote by $\pi: \sieg_g \rightarrow \A_g$ the natural projection map.

\begin{defin}
	 An algebraic subvariety $Z$ of $\A_g$ is  \emph{totally geodesic} if $Z=\pi(X)$ for some totally geodesic submanifold $X\subset \sieg_g$. 
	We say that $X$ is the symmetric space \textbf{{uniformizing}} $Z$.	
\end{defin}
In \cite{cfg,dejong-zhang,fgp,fpp,fpi,
	fp,fgs,gm1,gpt,hain,liu-yau-ecc,lz,moonen-special, moonen-oort} totally geodesic subvarieties of $\A_g$ have been studied in relation with the Coleman-Oort conjecture, which for large genus precludes the existence of Shimura subvarieties $Z$ of $\A_g$ generically contained in $j(\M_g)$, i.e such that $Z\subset \T_g$ and $Z\cap j(\M_g)\neq \vacuo$. (See \cite{coleman,oort-can} and \cite{moonen-oort} for a thorough survey). In our context, we will make use of the following characterization of \emph{Shimura} subvarieties of $\A_g$ due to Mumford and Moonen (see \cite{mumford-Shimura,moonen-linearity-1}): an algebraic subvariety of
$\A_g$ is Shimura if and only
if it is totally geodesic and contains a CM point.

\subsection{} In this paper we study the low genus  counterexamples to the Coleman-Oort conjecture. These are all in genus $g\leq 7$ and are obtained studying families of Galois covers of curves. The idea of the construction is the following. Take the family of all Galois covers  $C\rightarrow C'= C/G$, with fixed genera $g(C)=g$, $g(C')=g'$,  number of ramification points, and monodromy. Let $Z$ denote the closure in $\A_g$ of the locus described by $[JC]$ for $C$ varying in the family. The group $G$ acts holomorphically on every curve $C$ of the family, so it maps injectively into $Sp(2g, \erre)$. Denote by $\Ga$ the image of $G$ in $Sp(2g, \erre).$ Then $\pi(\sieg^\Ga)$ is a Shimura subvariety of $\A_g$ and, under the assumption that $\dim (S^2H^0(C, K_C))^G=\dim H^0(C, 2K_C)^G,$ we have $Z=\pi(\sieg^\Ga).$ In particular, it is a Shimura subvariety of $\A_g$ generically contained in $j(\M_g)$.

\subsection{} Our aim is to determine which symmetric space uniformizes each of the counterexamples. 
%The purpose of this paper is to determine which symmetric space uniformizes each of the counterexamples.
Since all $1-$dimensional irreducible hermitian symmetric spaces of the non-compact type are isomorphic to the Poincaré disc, we will focus on counterexamples of dimension greater than one. Table 2 in \cite{fgp} lists all counterexamples in genus $\leq 9$ obtained as above. {Denote by $B_n(\mathbb{C})=\{w\in \mathbb{C}^n:\ \sum_{k=1}^{n}w^i \overline{w}^i<1\}$ the open unit ball in $\mathbb{C}^n$, and by $\sieg_n= Sp(2n, \erre)/U(n)$ the Siegel space.} Our main result is the following:

\begin{teo}\label{teorema1}
  The uniformizing symmetric spaces for the examples in \cite{fgp} of dimension
   $\geq 2$ are the following:

{\begin{center}
	\begin{tabular}[H]{ccccc}
		
		& $g$  & G  & $\dim_{\mathbb{C}}$ & $\sieg^{\Ga}$\\
		\midrule
		(2)&  2 & $\mathbb{Z}/2 $ & 3 & $\sieg_2$\\
		(6) & 3 & $\mathbb{Z}/3 $ & 2 & $B_2(\mathbb{C})$ \\
		(8) & 3 & $\mathbb{Z}/4$ & 2&  $B_2(\mathbb{C})$  \\
		(10) & 4 & $\mathbb{Z}/3$ & 3 & $B_3(\mathbb{C})$ \\
		(14) & 4 & $\mathbb{Z}/6$  & 2 & $B_2(\mathbb{C})$ \\
		(16) & 6 & $\mathbb{Z}/5$ & 2 & $B_2(\mathbb{C})$ \\
		(26) & 2 & $\mathbb{Z}/2\times \mathbb{Z}/2$ & 2 &  $B_1(\mathbb{C})\times B_1(\mathbb{C})$ \\
		(27) & 3 & $\mathbb{Z}/2\times \mathbb{Z}/2$ & 3 & $B_1(\mathbb{C})\times B_1(\mathbb{C})\times B_1(\mathbb{C})$  \\
		(31) & 3& $S_3$ & 2 & $B_1(\mathbb{C})\times B_1(\mathbb{C})$ \\
		(32) & 3 & $D_4$  & 2 &  $B_1(\mathbb{C})\times B_1(\mathbb{C})$ \\
		\bottomrule
	\end{tabular}
\end{center}}
\end{teo} 
\ \\
{Shimura varieties obtained via Galois covers of elliptic curves were constructed in \cite{fpp}. There are 6 such families: (1e), (2e) (3e), (4e), (5e), (6e).
Only families (2e) and (6e) give raise to new Shimura varieties. The other four yield Shimura varieties which can be obtained also using coverings of $\mathbb{P}^1$.
\begin{teo}
	The uniformizing symmetric spaces for the examples in \cite{fpp} of dimension  $\geq 2$ are the following:
	\begin{center}
		\begin{tabular}[H]{ccccc}
			
			& $g$ & G  & $\dim_{\mathbb{C}}$ & $\sieg^\Ga$\\
			\midrule
			(2e)&  3 & $\mathbb{Z}/2 $ & 4 & $B_1(\mathbb{C})\times \sieg_3$  \\
			(6e) & 4 & $\mathbb{Z}/3 $ & 3 & $B_1(\mathbb{C})\times B_2(\mathbb{C})$  \\
			
			\bottomrule
		\end{tabular}
	\end{center}
\end{teo}}

\subsection{}The 6 families of Galois covers of elliptic curves have been further studied in \cite{fgs}, where it is shown that these families admit two fibrations in totally geodesic subvarieties, countably many of which are Shimura. If $f:C\rightarrow C'$ is a Galois covering of an elliptic curve $C'$ with group $G$, these fibration are given by the maps
\begin{align*}
P:\ [C\rightarrow C'] \mapsto &  Prym(C,C')\in \mathcal{A}^{\delta}_{g-1}\\
\phi :\ [C\rightarrow C'] \mapsto& [JC']\in \mathcal{A}_1.
\end{align*}
%$$P:\ [C\rightarrow C'] \mapsto Prym(C,C')\in \mathcal{A}_{g-1}$$
%$$\phi :\ [C\rightarrow C'] \mapsto [JC']\in \mathcal{A}_1.$$
%\begin{prop}\label{Irene}
%	Families (1e), (2e), (3e), (4e), (6e) are fibered in totally geodesic curves via their Prym maps and are fibered in totally geodesic subvarieties of codimension 1 via the map $\phi$. Therefore they contain infinitely many totally geodesic subvarietis and countably many Shimura subvarieties. The Prym map of family (5e) is constant.
%\end{prop}
Linked to the study of these maps is the decomposition, up to isogeny, of the Jacobian $JC$ of $C$, as $JC \sim JC' \times Prym(C,C')$. In the second part of this paper we study this decomposition at the level of the Siegel space. More precisely, we prove the following result.

\begin{teo}\label{teo2}
	%The uniformazing symmetric space $\sieg^\Ga$ relative to families (1e), (2e) (3e), (4e), (6e), admit a decomposition as $\Lambda \times M$, where $\Lambda$ is of type AIII(1,1) and $M$ is irreducible of codimension $1$.
	%Then the image in $\A_g$ of the factor of type A III (1,1) of $\sieg^\Ga$, which is irreducible of dimension $1$, is an irreducible component of the fiber of the Prym map, whereas the image in $\A_g$ of the factor $M$ of $\sieg^\Ga$, which is irreducible of codimension $1$, is an irreducible component of the of the fiber of $\phi$.
	
	Let $\sieg^\Ga$ be the uniformizing symmetric space associated to one of the families (1e), (2e) (3e), (4e), (6e). Then\\
	\begin{enumerate}
		\item [i)] $\sieg^\Ga$ decomposes as $B_1(\mathbb{C}) \times M$, where $M$ is an hermitian symmetric space of codimension $1$.\\
		\item [ii)] The image in $\A_g$ of the factor $M$ of $\sieg^\Ga$ is an irreducible component of the fiber of $\phi$. In particular, $M$ uniformizes the irreducible components of the fibers of $\phi$.\\
		\item [iii)] The image in $\A_g$ of $B_1(\mathbb{C})$ is an irreducible component of the fiber of the Prym map. \ \\	
	\end{enumerate}
\end{teo}

\begin{center}
\textsc{Acknowledgements}	
\end{center}
The author would like to thank her advisor Alessandro Ghigi, firstly, for proposing her the topic of this work and, secondly, for his help and support, which have been fundamental for its realization. She also would like to thank Paola Frediani for interesting discussions.

\section{Cartan decomposition}
\subsection{}\label{tabella}The following table exhausts the list of irreducible Hermitian symmetric spaces of the non-compact type. (See \cite[Table V, p. 518]{helgason}).

\begin{center}

\begin{tabular}[H]{cccc}
	
	 & Space &  Complex Dimension & \\
	\midrule
	A III & $SU(p,q)/S(U(p)\times U(q))$	&  $pq$ &$ 1\leq p \leq q$	 \\
	BD I $(q=2)$	&$SO^0(p,2)/SO(p)\times SO(2)$	& $p$ & $p\geq 4$\\
	D III	& $SO^*(2n)/U(n)$ & 	$n(n-1)/2$& $n\geq 4$\\

	C I	& $Sp(2n, \erre)/U(n)$	 &	$n(n+1)/2$ & $n\geq 2$\\
	E III	& 		& $16$ &\\
	E VII	& 	& $27$ &\\

	\bottomrule

\end{tabular}

\end{center}
{The conditions on the parameters in the last column avoid overlaps in the Table. In fact, for small dimensions, there are some coincidences between different classes (see \cite[p. 519]{helgason}). We point out that  C I $(n)= \sieg_n$, and that it is possible to identify A III $(1,n)$ with the unit ball $B_n(\mathbb{C})$ in $\mathbb{C}^n$ (see e.g. \cite[Ex. 10.7, pp. 282-285]{knvol2}). }
%298
%173

\subsection{}\label{sec:shimura-dei-gruppi}

%Sia $G:= \Sp(V, Q)$ e 
%sia $\sieg = \sieg (V,Q)$. 
%Fissiamo un gruppo finito $\Ga \subset G$ e consideriamo $\sieg^\Ga$
Fix a rank $2g$ lattice $\Lambda$ and an alternating form $Q: \Lambda \times \Lambda \longrightarrow \mathbb{Z}$ of type $(1,...,1)$. Let $G:= \Sp(\Lambda_\erre, Q)$ and $\sieg=\sieg(\Lambda_\erre, Q)$ the Siegel space, which can be defined as: \begin{equation}\label{siegel}
\sieg=\{J\in GL(\Lambda_\erre):\ J^2=-I,\ J^*Q=Q,\ Q(x, Jx)>0,\ \forall x\neq 0\}.
\end{equation}
Fix a finite subgroup $\Ga \subset G$ and consider $\sieg^\Ga$. 
%\Blem Let $(M,g)$ be a Riemannian symmetric space of the noncompact type. Let $H$ be a group acting isometrically on $(M,g)$. If $M^H$ is nonempty, then it is a smooth connected submanifold of $M$.\Elem
{%See \cite[Lemma 3.3]{fgp} for a proof.
As a consequence  of the Cartan fixed point theorem, $\sieg^\Ga$ is nonempty and hence it is a connected complex submanifold of $\sieg$ (see e.g. \cite[Lemma 3.3]{fgp}).  For $J\in\sieg$, denote by $\g=\k\oplus \p$ the Cartan decomposition associated to $\sieg$ and by $K=Stab_J(G)$. Set $\mathfrak{Z}_\g(\Ga)=\{X\in \g: \Ad_\Ga(X)=X\}$, and define similarly $\Z_{\k}(\Gamma)$ and $\Z_{\p}(\Gamma)$.}  For $J\in \sieg^\Ga$, $T_J\sieg^\Ga=(T_J\sieg)^\Ga=\Z_\p(\Ga).$

\Blem\label{Centro p} Let $J\in \sieg^\Ga.$ Then $\Z_\g(\Ga)=\Z_\k(\Ga)\oplus \Z_\p(\Ga)$. \Elem  
\Bdim Let $X\in \Z_\g(\Ga),$ and write $X=u+v$ with $u\in \k$ and $v\in \p$. Since $J\in \sieg^\Ga=\{J\in \sieg:\  \gamma J \gamma^{-1}=J \ \forall \gamma\in \Ga\}$, we have that $\Ga\subset Stab_J=K$. It follows that $\Ad_\Ga(\p)\subset \p$ and $\Ad_\Ga(\k)\subset \k$. Thus $X=\Ad_\Ga(X)=\Ad_\Ga(u)+\Ad_\Ga(v)$, with $\Ad_\Ga(u)\in \k$ and $\Ad_\Ga(v)\in \p$, and we conclude that $u=\Ad_\Ga(u)$ and $\Ad_{\Ga}(v)=v$. \Edim 

%\begin{cor}
%$\sieg^\Ga$ is a totally geodesic submanifold of $\sieg$. 	
%\end{cor}

%\begin{proof}
%	For $J\in \sieg^\Ga$, $T_J\sieg^\Ga=(T_J\sieg)^\Ga=\p^\Ga=\Z_\p(\Ga)=\Z_\g(\Ga)\cap \p.$ An immediate consequence is that $\p^\Ga\subset \p$ is a Lie triple system and hence $\sieg^\Ga$ is totally geodesic in $\sieg$.
%\end{proof} 
Observe that, since $\Ga$ is a group of isometries of $\sieg$, $\sieg^\Ga$ is a totally geodesic submanifold of $\sieg$. In particular, it is an hermitian symmetric space of the non-compact type. The image of $\sieg^\Ga$ in $ \A_g$ is a Shimura variety (See e.g. \cite[Proposition 3.7]{fgp}).

%\Blem $\sieg^\Ga$ is a complex submanifold of $\sieg$. In particular, $\sieg^\Ga$ has no euclidean factor. \Elem
\Blem\label{noeuclidean} Let $Z$ be a complex totally geodesic submanifold of $\sieg$. Then $Z$ has no euclidean factor. \Elem

\Bdim Since $\sieg$ is an hermitian symmetric space, the complex structure $\widehat{I}$ on $\sieg$ is induced by $\Ad(z)|_\p: \p\rightarrow \p$ with $z\in Z(K)$. 
%Thus, since $\Ga\subset K$, for $X\in \Z_\g(\Ga)$ we have $Ad(\gamma)(Ad(z)X)=Ad(\gamma z)X=Ad(z\gamma)z=Ad(z)Ad(\gamma)X=Ad(z)X$. Thus $Ad(z)(\Z_\g(\Ga))\subset \Z_\g(\Ga)$. This proves that $\sieg^\Ga$ is a complex submanifold of $\sieg$. 
More precisely, one can show that the matrix representing $z$ in an appropriate basis is given by
 $z=\frac{1}{\sqrt{2}}\left( \begin{matrix}
I & I \\ -I & I
\end{matrix}
\right)\in Sp(2g, \erre)\simeq G$. With this choice it easy is to see that for $X\in \p, X\neq 0$, we have $[X, \widehat{I}X]\neq 0$. \Edim

\Blem\label{cartandec} Let $(G,K)$ be a symmetric pair with no euclidean factor and Cartan decomposition $\g=\k\oplus \p$. Suppose that $G$ acts almost effectively on the coset space $M=G/K$. Then $\k=[\p,\p]$. \Elem 
(See e.g. \cite[Theorem 4.1, p. 243]{helgason}).
As a consequence of Lemma \ref{noeuclidean}, $\sieg^\Ga$ has no euclidean factor. Thus, we have the following Corollary of Lemma \ref{cartandec}.

\begin{cor}\label{decomposizione}
If $(G',K')$ is an almost effective symmetric pair associated to $\sieg^\Ga$, then its Cartan decomposition $\g'=\k'\oplus \p'$ is given by  \begin{equation*}
\p'=\Z_\p(\Ga), \quad \quad \k'=[\Z_\p(\Ga), \Z_\p(\Ga)].
\end{equation*} 
\end{cor}
Notice that since $\sieg^\Ga$ is of the noncompact type, the Cartan decomposition in Corollay \ref{decomposizione} determines the symmetric space up to isomorphism.

\section{Uniformizing symmetric spaces}
\subsection{}The known counterexamples to the Coleman-Oort conjecture can be divided in three classes:\begin{enumerate}
	\item those obtained as families of Galois covers of $\mathbb{P}^1$;
	\item those obtained as families of Galois covers of elliptic curves;
	\item those obtained via fibrations constructed on the examples in $(2)$.
\end{enumerate}
There are some overlaps between $(1)$ and $(2)$, and we will prefer the description $(2)$ as more economical. Therefore we will first compute the uniformizing symmetric space for the examples belonging to only $(1)$, next for those in $(2)$, including $4$ examples that also appear in $(1)$. Finally, we will discuss the relationship between the fibrations and the uniformizing symmetric space of the examples in $(2)$. This will show how to uniformize the examples in $(3)$. \\

The examples in $(1)$ and $(2)$ are all constructed as follows: let $\mathcal{C}\rightarrow B$ be the family of all Galois covers $C_t\rightarrow C'_t=C_t/G$, with fixed genera $g=g(C_t)$, $g'=g(C_t')$, ramification and monodromy. Let $Z$ denote the closure in $\A_g$ of the locus described by $[JC_t]$ for $C_t$ varying in the family. Then, under a numerical condition, $Z$ provides an example of a special subvariety of $\A_g$ contained in the Torelli locus, whose uniformizing symmetric space is $\sieg^\Ga$. Table 2 in \cite{fgp} lists all examples with $g'=0$ and $g\leq 9$. Table 2 in \cite{fpp} lists all examples with $g'=1$ and $g\leq 9$.\\

%{For $r\geq 0$ and $g'\geq 0$ set $\Gamma_{g',r}:=\sx \alpha_1,\beta_1,..., \alpha_{g'}, \beta_{g'}, \gamma_1,...,\gamma_r |\  \prod_{1}^{r}\gamma_i  \cdot\  \prod_1^{g'} [\alpha_j,\beta_j]=1 \xs.$ A datum is a triple $(\textbf{m}, G, \theta)$, where $\textbf{m}:=(m_1,...,m_r)$ is an $r-$tuple of integers $m_i\geq 2$, $G$ is a finite group and $\theta: \Gamma_{g',r}\rightarrow G$ is an epimorphism such that $\theta(\gamma_i)$ has order $m_i$ for each $i$. Fix a compact Riemann surface $C'$ of genus $g'\geq 0$. To any datum is associated a family $f:C\rightarrow C'=C/G$ of Galois covers of $C'$ branched on $r$ distinct points. See \cite[\S 2]{fpp} and references therein for a detailed description of this construction. \\}

If $f: C \rightarrow C'$ is one of the coverings, the action of $G$ on $C$ induces the following action of $G$ on holomorphic $1-$forms
$$\rho: G \rightarrow GL(H^0(C, K_{C})), \quad \quad \rho(g)(\omega)=g.\omega:=(g^{-1})^*(\omega).$$
Notice that the equivalence class of $\rho$ only depends on the family.
The homomorphism $\rho$ maps $G$ injectively into $Sp(\Lambda, Q)$, where $\Lambda=H_1(C, \mathbb{Z})$ and $Q$ is the cup form (see e.g. \cite[p. 270]{farkaskra}). Denote by $\Gamma$ the image of $G$ in $Sp(\Lambda, Q)$ and $\sieg = \sieg (\Lambda_\erre,Q)$, defined as in \eqref{siegel}.

%Let $Z(\textbf{m}, G, \theta)$ denote the closure in $\A_g$ of the locus described by $[JC]$ for $C$ varying in the family. Then, under a numerical condition, $Z$ provides an example of a special subvariety of $\A_g$ contained in the Torelli locus, whose uniformizing symmetric space is $\sieg^\Ga$. (See $\S 1$ and references therein.) This section will be devoted to the study of uniformizing symmetric spaces of all special subvarieties of $\A_g$ in genus $\leq 9$ obtained using this method. 

%\Blem\label{simplettico in base ortonormale} Let $X$ be a compact Riemann surface of genus $g$ and $\omega_1,...,\omega_g$  a unitary basis of $H^0(K_X)$ with respect to the hermitian product $\sx \omega, \omega' \xs:=i \int_X \omega\wedge \overline{\omega'}$. Let $\g=\mathfrak{sp}(H^1(X, \erre), Q)\simeq \mathfrak{sp}(2g, \erre)$, where $Q$ is the cup form. An element $U\in \mathfrak{gl}(H^1(X, \mathbb{C}))$  belongs to $\g$ if and only if  the matrix representing $U$ in the basis $\omega_1,...,\omega_g, \overline{\omega_1},...,\overline{\omega_g}$ of $H^1(X, \mathbb{C})$ is of the form $U=\left( \begin{matrix}
%C & \overline{D} \\ D & \overline{C}
%\end{matrix}
%\right)$ with $D=D^t$ and $C\in \mathfrak{u}(g)$. \Elem

\subsection{Galois coverings of the line. Cyclic case.}
Fixed a family of covers, let $a_i$ be an element of order $m_i$ in $G$ that represents the local monodromy of the covering $\pi: C \rightarrow \mathbb{P}^1$ at the $i-$th branch point. In the case the group $G$ is cyclic, the study of the multipilicity of a given irreducible representation of $G$ in $H^0(C, K_C)$ reduces to the study of the eigenspaces of the generator of the group. The following is the Chevalley-Weil formula in the cyclic case.

\Blem\label{Moonen} Let $\mathbb{Z}/m=\sx\zeta\xs$ be the cyclic group of order $m$. Consider the $\mathbb{Z}/m-$cover $\pi: C_t \rightarrow \mathbb{P}^1, $ with $t=(t_1,...,t_N)$ branch points in $\mathbb{P}^1$ with local monodromy $a_i$ about $t_i$. For $n\in \mathbb{Z}/m$ we write $H^0(C_t, K_{C_t})_{(n)}:=\{\omega \in H^0(C_t, K_{C_t}):\ \zeta.\omega=\zeta^{n}\omega\}.$ Then 
$$\dim H^0(C_t, K_{C_t})_{(n)}=
-1+\sum_{i=1}^{N}  \frac{[na_i]_m}{m},$$ 
where $[a]_m$ denotes the unique representative of $a\in \mathbb{Z}/m$ in $\{0,...,m-1\}$.
\Elem 
Fixing a unitary basis of $H^0(C, K_{C})_{(n)}$ for each $n\in \mathbb{Z}/m$, one obtains a unitary basis $\{\omega_1,...,\omega_g\}$ of $H^0(C, K_{C})$ with respect to which the matrix $A_0$ that represents the generator $\rho(\zeta)$ of $\Gamma$ is diagonal. The multiplicity of each $m-$th root of unity on the diagonal is given by Lemma \ref{Moonen}. Consider now the pullback action $\tilde \rho: G \rightarrow GL(H^1(C, \mathbb{C}))$ of $G$ on the whole $H^1(C, \mathbb{C})$. Notice that if $\omega\in H^0(C, K_{C})$, then $\tilde \rho(\gamma)(\overline{\omega})=\overline{\rho(\gamma)(\omega)}$. Thus, with respect to the basis $\{\omega_1,..., \omega_g, \overline{\omega_1}, ..., \overline{\omega_g}\}$\label{ciao} of $H^1(C, \mathbb{C})$, $\tilde\rho(\zeta)$ is represented by the matrix
\begin{equation}\label{A}
A=\left( \begin{matrix}
A_0 & 0 \\ 0 & \overline{A_0}
\end{matrix}
\right).
\end{equation}
%\begin{prop}
%	Let $\g'=\p'\oplus \k'$ be the Cartan decomposition of the uniformizing symmetric space $\sieg^\Ga$ associated to a family of cyclic covering of $\mathbb{P}^1$. Denote by $A$ the generator of $\Gamma\subset GL(H^1(C, \mathbb{C}))$. Then $\p'= \Z_{\p}(A)$ and $\k'=[\p',\p']$.
%\end{prop}
%\begin{proof}	
% As proved in [Shimura sub. in the Torelli locus], denoted by $Z$ the closure in $\A_g$ of the locus described by $[JC]$, then $Z=\pi(\sieg^\Ga)$. Since $\Ga=< A>$, then $$\sieg^\Ga= \sieg^A=\{J\in\sieg:\ JA=AJ\}.$$
%By \cite[Theorem 3.9]{fgp} it follows that
%$Z=\pi(\sieg^{\Ga})=\pi(\sieg^A)$, where $\sieg^A=\{J\in\sieg:\ JA=AJ\}$. 
Clearly, $\sieg^\Ga=\sieg^A=\{J\in\sieg:\ JA=AJ\}$. For $J\in \sieg^\Ga$, denote by $\g=\k\oplus\p$ the Cartan decomposition associated to $\sieg$. Fix the unitary basis $\{\omega_1,...,\omega_g, \overline{\omega_1},...,\overline{\omega_g}\}$ of $H^1(C, \mathbb{C})$ defined above. Observe that, since $\sx \omega, \omega'\xs =i Q(\omega, \overline{\omega'})$, the matrix representing $Q$ in this basis is $Q=\left( \begin{matrix}
0 & iI_g \\ -iI_g & 0
\end{matrix}
\right).$ An element $U\in \mathfrak{gl}(H^1(X, \mathbb{C}))$ belongs to $\g\simeq \sp(2g, \erre)$ if and only if it is real, and thus of the form $\left( \begin{matrix}
C & \overline{D} \\ D & \overline{C}
\end{matrix}
\right)$, and satisfies $$0=U^t Q+QU=i \left( \begin{matrix}
D-D^t &  C^t+\overline{C} \\ -\overline{C}^t-{C} & \overline{D}^t-{D}
\end{matrix}
\right).$$
Thus $D=D^t$ and $C\in \mathfrak{u}(g)$ and one immediately gets the Cartan decomposition \begin{equation}\label{dec}
U=\left( \begin{matrix}
C & 0 \\ 0 & \overline{C}
\end{matrix}
\right)+\left( \begin{matrix}
0 & \overline{D} \\ D & 0
\end{matrix}
\right), \quad \quad \text{with }\left( \begin{matrix}
C & 0 \\ 0 & \overline{C}
\end{matrix}
\right)\in \k,\quad \left( \begin{matrix}
0 & \overline{D} \\ D & 0
\end{matrix}
\right) \in \p.
\end{equation} Now $U\in \Z_{\g}(A)$ if and only if $UA=AU$, i.e. 
$$\left( \begin{matrix}
C & \overline{D} \\ D & \overline{C}
\end{matrix}
\right)\left( \begin{matrix}
A_0 & 0 \\ 0 & \overline{A_0}
\end{matrix}
\right)=\left( \begin{matrix}
A_0 & 0 \\ 0 & \overline{A_0}
\end{matrix}
\right)\left( \begin{matrix}
C & \overline{D} \\ D & \overline{C}
\end{matrix}
\right).$$
In other words, $U$ must preserve the eigenspaces of $A$. This concludes the study of $\sieg^\Ga$. Indeed, if $\sieg^\Ga$ has Cartan decomposition $\g'=\p'\oplus \k'$, it follows from Corollary \ref{decomposizione}, that $\k'=[\p', \p']$ and
$$\p'= \Z_{\p}(A)=\Z_{\g}(A)\cap \p=\{\left( \begin{matrix}
0 & \overline{D} \\ D & 0
\end{matrix}
\right), \ D=D^t,\ DA_0=\overline{A_0}D  \}.$$
%\end{proof}

We present in the following the analysis of the uniformizing symmetric space of the various examples of Galois coverings of the line. This is summarized in Theorem \ref{teorema1}. Throughout the discussion, we will make use of the notation \eqref{dec}.\\
\paragraph{\textbf{{Notation:}}} We will denote by $M(n,m, \C)$ the space of $n\times m$ complex matrices.\\ 
\paragraph{\textbf{Family (8)}}
$A_0= \diag (\zeta^3,\zeta^3,\zeta)$, with $\zeta=e^{2\pi i /4}$.\\ $U=\left( \begin{matrix}
C & \overline{D} \\ D & \overline{C}
\end{matrix}
\right)\in \Z_{\g}(A)$ iff $ C=\left( \begin{matrix}
E & 0\\ 0 & i\lambda
\end{matrix}
\right),$ with $ E\in \u(2),\ \lambda\in \erre,$ and $ D=\left( \begin{matrix}
0 & d \\ d^t & 0
\end{matrix}
\right),\ d\in M(2,1,\mathbb{C}).$\\
$\p'= \Z_{\p}(A)=\{\left( \begin{matrix}
0 & \overline{D} \\ D & 0
\end{matrix}
\right),\ D=\left( \begin{matrix}
0 & d \\ d^t & 0
\end{matrix}
\right),\ d\in M(2,1,\mathbb{C})\}.$\\
Let $X, Y\in \p'$, $X=\left( \begin{matrix}
0 & \overline{D} \\ D & 0
\end{matrix}
\right)$ and $Y=\left( \begin{matrix}
0 & \overline{E} \\ E & 0
\end{matrix}
\right)$. Then
$$[X, Y]=XY-YX=D=\left( \begin{matrix}
\overline{D}E-\overline{E}D & 0 \\ 0 & D\overline{E}-E\overline{D}
\end{matrix}
\right).$$
If $D=\left( \begin{matrix}
0 & d \\ d^t & 0
\end{matrix}
\right),$ with $ d\in M(2,1,\mathbb{C})$ and $E=\left( \begin{matrix}
0 & e \\ e^t & 0
\end{matrix}
\right),$ with $e\in M(2,1,\mathbb{C})$, we have 
$$\overline{D}E-\overline{E}D= \left( \begin{matrix}
\overline{d}e^t-\overline{e}d^t & 0 \\ 0 & \overline{d}^te-\overline{e}^td
\end{matrix}
\right).$$
Set $F=\overline{d}e^t-\overline{e}d^t$ and notice that $\overline{d}^te-\overline{e}^td=tr(F)=2i Im (\sx e, d \xs)$, where $\sx\ ,\ \xs$ denotes the standard hermitian product in $\mathbb{C}^2$. Thus, recalling Corollary \ref{decomposizione}, we conclude 
$$\k'=[\p', \p']\subset\{ \left( \begin{matrix}
C & 0 \\ 0 & \overline{C}
\end{matrix}
\right), \ C=\left( \begin{matrix}
F & 0\\ 0 & \tr(F)
\end{matrix}
\right),\ F\in \u(2)\}.$$
The choices $\{d=(-i/2, 0)^t$, $e=(1, 0)^t\}$,  $\{d=(0, -i/2)^t$, $e=(0, 1)^t\}$, $\{d=(0, 1)^t$, $e=(-1, 0)^t\}$, $\{d=(0, 1)^t$, $e=(i, 0)^t\}$ for $D$ and $E$, give as $F$ the following matrices: 
$$\left( \begin{matrix}
i & 0\\ 0 & 0
\end{matrix}
\right), \quad \left( \begin{matrix}
0 & 0\\ 0 & i
\end{matrix}
\right), \quad \left( \begin{matrix}
0 & 1\\ -1 & 0
\end{matrix}
\right), \quad \left( \begin{matrix}
0 & i\\ i & 0
\end{matrix}
\right)$$
which constitute a basis for $\u(2)$. Thus the equality sign holds:
$$\k'=[\p', \p']=\{ \left( \begin{matrix}
C & 0 \\ 0 & \overline{C}
\end{matrix}
\right), \ C=\left( \begin{matrix}
F & 0\\ 0 & \tr(F)
\end{matrix}
\right),\ F\in \u(2)\}.$$
Consider now the adjoint representation $\ad: \k' \rightarrow \mathfrak{gl}(\p')$ of $\k'$ on $\p'$. If $F\in \k'=\u(2)$, and $d\in \mathbb{C}^2=\p'$, then $\ad_F(d)=\overline{E}d-Tr(E)d$. One easily checks it is an irreducible representation and this proves that $\sieg^A$ is an irreducible hermitian symmetric space of complex dimension $2$. We conclude that $\sieg^\Ga$ is of type A III $(p=1, q=2)$. \\

% Since $\sieg^A$ has complex dimention $2$, as irreducible hermitian symmetric space it has to be either of type A III or of type BD I: $$\text{A III:}\quad SU(1,2)/S(U(1)\times U(2)) \quad \quad \quad \text{BD I:}\quad SO^0(2,2)/SO(2)\times SO(2).$$

%We want to apply Lemma \ref{Centro p} to obtain $\Z_{\p}(\Ga)=\Z_{\p}(A)$.

\paragraph{\textbf{Family (10)}}
$A_0= \diag (\zeta^2,\zeta^2,\zeta^2, \zeta)$, with $\zeta=e^{2\pi i /3}$.\\
$\p'= \Z_{\p}(A)=\{\left( \begin{matrix}
0 & \overline{D} \\ D & 0
\end{matrix}
\right),\ D=\left( \begin{matrix}
0 & d \\ d^t & 0
\end{matrix}
\right),\ d\in M(3,1,\mathbb{C})\}.$\\
Let $X, Y\in \p'$, $X=\left( \begin{matrix}
0 & \overline{D} \\ D & 0
\end{matrix}
\right)$ and $Y=\left( \begin{matrix}
0 & \overline{E} \\ E & 0
\end{matrix}
\right)$. Then
$$[X, Y]=XY-YX=\left( \begin{matrix}
\overline{D}E-\overline{E}D & 0 \\ 0 & D\overline{E}-E\overline{D}
\end{matrix}
\right).$$
If $D=\left( \begin{matrix}
0 & d \\ d^t & 0
\end{matrix}
\right),$ with $ d\in M(3,1,\mathbb{C})$ and $E=\left( \begin{matrix}
0 & e \\ e^t & 0
\end{matrix}
\right),$ with $e\in M(3,1,\mathbb{C})$, we have 
$$\overline{D}E-\overline{E}D= \left( \begin{matrix}
\overline{d}e^t-\overline{e}d^t & 0 \\ 0 & \overline{d}^te-\overline{e}^td
\end{matrix}
\right).$$
Set $F=\overline{d}e^t-\overline{e}d^t$ and notice that $\overline{d}^te-\overline{e}^td=tr(F)=2i Im (\sx e, d \xs)$, where $\sx\ ,\ \xs$ denotes the standard hermitian product in $\mathbb{C}^3$. Thus, similarly to the previous family, we get
$$\k'=[\p', \p']=\{ \left( \begin{matrix}
C & 0 \\ 0 & \overline{C}
\end{matrix}
\right), \ C=\left( \begin{matrix}
F & 0\\ 0 & \tr(F)
\end{matrix}
\right),\ F\in \u(3)\}.$$
In particular $\dim \k' = 9.$ Looking at the adjoint representation $\ad: \k' \rightarrow \mathfrak{gl}(\p')$, one checks that $\sieg^A$ is irreducible. Thus it is an irreducible hermitian symmetric space of dimension $3$. Looking at $\dim\k'$ in the Table of section  \ref{tabella} we conclude that $\sieg^\Ga$ is of type A III $(p=1, q=3)$.\\

\paragraph{\textbf{Family (2)}}
Here ${A_0=-I_2}$. Thus $\sieg^\Ga=\sieg^A=\sieg=\sieg_2$. In fact here $Z=\overline{\M_2}=\A_2$. \\

\paragraph{\textbf{Family (6)}}
$A_0= \diag( 	\zeta^2 ,
	 \zeta^2 , \zeta )$, with $\zeta=e^{2\pi i /3}$.\\ 
$\p'=\{\left( \begin{matrix}
0 & \overline{D} \\ D & 0
\end{matrix}
\right),\ D=\left( \begin{matrix}
0 & d \\ d^t & 0
\end{matrix}
\right),\ d\in M(2,1,\mathbb{C})\}.$\\
$\k'=[\p', \p']=\{ \left( \begin{matrix}
C & 0 \\ 0 & \overline{C}
\end{matrix}
\right), \ C=\left( \begin{matrix}
F & 0\\ 0 & \tr(F)
\end{matrix}
\right),\ F\in \u(2)\}.$\\ $\sieg^\Ga$ is of type A III $(p=1, q=2)$.\\

\paragraph{\textbf{Family (14)}}
$A_0=\diag(\zeta^5, \zeta^5, \zeta^2, \zeta)$, with $\zeta=e^{2\pi i /6}$.\\ $\p'=\{\left( \begin{matrix}
	0 & \overline{D} \\ D & 0
\end{matrix}
\right),\ D=\left( \begin{matrix}
	0 &0  &d \\ 0 & 0 & 0\\d^t & 0 & 0
\end{matrix}
\right),\ D\in M(4,\mathbb{C}),\ d\in M(2,1,\mathbb{C})\}$. \\
$\k'=[\p', \p']=\{ \left( \begin{matrix}
C & 0 \\ 0 & \overline{C}
\end{matrix}
\right), \ C=\left( \begin{matrix}
F & 0 & 0\\ 0 & 0 & 0\\ 0 & 0 &  \tr(F)
\end{matrix}
\right),\ F\in \u(2)\}.$\\ $\sieg^\Ga$ is of type A III $(p=1, q=2)$.\\

\paragraph{\textbf{Family (16)}}
$A_0=\diag(\zeta^4. \zeta^4, \zeta^4, \zeta^3, \zeta^3, \zeta^2)$, with $\zeta=e^{2\pi i /5}$.\\ $\p'=\{\left( \begin{matrix}
	0 & \overline{D} \\ D & 0
\end{matrix}
\right),\ D=\left( \begin{matrix}
0 & 0 & 0\\ 0 & 0 & d\\0& d^t & 0
\end{matrix}
\right),\ D\in M(6,\mathbb{C}),\ d\in M(2,1,\mathbb{C})\}$.\\
$\k'=[\p', \p']=\{ \left( \begin{matrix}
C & 0 \\ 0 & \overline{C}
\end{matrix}
\right), \ C=\left( \begin{matrix}
 0 & 0 & 0\\ 0 & F & 0\\ 0 & 0 &  \tr(F)
\end{matrix}
\right),\ F\in \u(2)\}.$\\ $\sieg^\Ga$ is of type A III $(p=1, q=2)$.
\subsection{Galois coverings of the line. Non-cyclic case.}
In the case the group $G$ is not cyclic, we cannot use Lemma \ref{Moonen} to study the action $\rho$ of $G$ on $H^0(C, K_C)$.  In particular we will use the general Chevalley-Weil formula to get the multiplicity of a given irreducible representation of $G$ in $H^0(C, K_C)$ (see e. g. \cite[Theorem 1.3.3]{gleissner}).\ \\

\paragraph{\textbf{Family (27)}}
This is the family of covers of $\mathbb{P}^1$ with group $G=\mathbb{Z}/2 \times \mathbb{Z}/2$, ramification data $\textbf{m}=(2^6)$, $g=3$ and dimension $3$. Label the characters of $\mathbb{Z}/2 \times \mathbb{Z}/2$ as follows:
\begin{center}
	\begin{tabular}[H]{c|cccc}
		
		& $(0,0)$ & $(0,1)$ & $(1,0)$ &  $(1,1)$ \\
		\midrule
		$\chi_1$ & 1 & 1 & 1 & 1\\
		$\chi_2$ & 1 & -1 & 1 & -1\\
		$\chi_3$ & 1 & 1 & -1 & -1 \\
		$\chi_4$ & 1 & -1 & -1 & 1\\
		\bottomrule
	\end{tabular}
\end{center}
\ \\\\
From the Chevalley-Weil formula follows that the character $\chi_\rho$ of the action $\rho$ of $G$ on $H^0(C, K_C)$ is given by $\chi_\rho=\chi_2+\chi_3+\chi_4$. Thus there exists $V_i\subset H^0(C, K_C)$, $i=2,3,4$, $\dim V_i=1$, such that $H^0(C, K_C)=V_2\oplus V_3\oplus V_4$ and $$\rho: G\rightarrow GL(V_2\oplus V_3\oplus V_4), \quad \rho(g)=\left( \begin{matrix}
\rho_2(g) & 0 & 0 \\ 0 & \rho_3(g) & 0 \\ 0 & 0 & \rho_4(g)
\end{matrix}
\right),$$
where $\rho_i$ denotes the irreducible representation of $G$ with character $\chi_i$. The choice of norm one vectors that span $V_i$, $i=2,3,4$, gives a unitary basis $\{\omega_1, \omega_2, \omega_3\}$ of $H^0(C,K_C)$, with respect to which $$\rho(0,1)=\left(\begin{matrix}
1 & 0 & 0 \\ 0 & -1 & 0 \\ 0 & 0 & -1
\end{matrix}\right)=:A_0, \quad \rho(1,0)=\left(\begin{matrix}
-1 & 0 & 0 \\ 0 & 1 & 0 \\ 0 & 0 & -1
\end{matrix}\right)=:B_0.$$
Consider now the action $\tilde \rho: G \rightarrow GL(H^1(C, \mathbb{C}))$ of $G$ on $H^1(C, \mathbb{C})$, and let $A, B$ denote the matrices that represent $\tilde{\rho}(1,0)$ and $\tilde{\rho}(0,1)$ with respect to the basis of $H^1(C, \mathbb{C})$ induced by $\{\omega_1, \omega_2\}$:
$$A=\left( \begin{matrix}
A_0 & 0 \\ 0 & \overline{A_0}
\end{matrix}
\right), \quad B=\left( \begin{matrix}
B_0 & 0 \\ 0 & \overline{B_0}
\end{matrix}
\right). $$
The matrices $A$ and $B$ are the generators of the injective image $\Ga$ of $H$ in $Sp(H^1(C,\mathbb{Z}), Q)$. Thus $\sieg^{\Ga}=\sieg^{A,B}=\{J\in\sieg:\ JA=AJ,\ JB=BJ\}$. With the notation \eqref{dec}, $U\in\g$ is of the form $U=\left( \begin{matrix}
C & \overline{D} \\ D & \overline{C}
\end{matrix}
\right)$, with $D=D^t$ and $C\in \mathfrak{u}(2)$, and $U\in \Z_{\g}(A, B)$ satisfies
$$C=\left( \begin{matrix}
i\lambda_1 & 0 & 0\\ 0 &  i\lambda_2 & 0\\
0 & 0 & i\lambda_3
\end{matrix}
\right), \ \lambda_i\in \erre \quad D=\left(\begin{matrix}
d_1 & 0 & 0 \\ 0 & d_2 & 0 \\ 0 & 0 &d_3
\end{matrix}\right),\ d_i\in \mathbb{C}.$$
Thus, from Lemma \ref{Centro p}, the Cartan decomposition associated to $\sieg^\Ga$ is $\g'=\p'\oplus \k'$, with 
$$\p'= \Z_{\p}(A)=\{\left( \begin{matrix}
0 & \overline{D} \\ D & 0
\end{matrix}
\right),\ D=\diag(d_1,d_2,d_3),\ d_i\in \mathbb{C}\}$$ and $\k'=[\p',\p']$. To calculate $\k'$, consider $X, Y\in \p'$. If  $X=\left( \begin{matrix}
0 & \overline{D} \\ D & 0
\end{matrix}
\right)$ and $Y=\left( \begin{matrix}
0 & \overline{E} \\ E & 0
\end{matrix}
\right)$, with $D=\left( \begin{matrix}
d_1 & 0 \\ 0 & d_2
\end{matrix}
\right),$ and $E=\left( \begin{matrix}
e_1 & 0 \\ 0 & e_2
\end{matrix}
\right),$ we have 
$$[X, Y]=\left( \begin{matrix}
\overline{D}E-\overline{E}D & 0 \\ 0 & D\overline{E}-E\overline{D}
\end{matrix}
\right),$$
where $$\overline{D}E-\overline{E}D=\diag(\overline{d_1}e_1-\overline{e_1}d_1,\overline{d_2}e_2-\overline{e_2}d_2, \overline{d_3}e_3-\overline{e_3}d_3  ).$$
Thus 
$\k'=[\p', \p']=\Z_{\g}(A)\cap \k=\{\left( \begin{matrix}
C & 0 \\ 0 & \overline{C}
\end{matrix}
\right), \ C=\diag(i\lambda_1, i\lambda_2, i\lambda_3), \lambda_i\in\erre\}.$
In this case, the adjoint representation $\ad: \k' \rightarrow \mathfrak{gl}(\p')$ is not irreducible. Indeed, let $C=\diag(i\lambda_1, i\lambda_2)\in \k'$ and $D=\diag(d_1, d_2)\in \p'$. We have
$$ \left[\left(\begin{matrix}
C & 0 \\ 0 & \overline{C}
\end{matrix}\right),  \left(\begin{matrix}
0 & \overline{D} \\ D & 0
\end{matrix}\right)    \right] = \left(\begin{matrix}
0 & C\overline{D}-\overline{D}\overline{C} \\ \overline{C}D-DC
\end{matrix}\right).$$
Thus $\ad_C(D)=\overline{C}D-DC=\diag(-2i\lambda_1d_1, -2i\lambda_2d_2, -2i\lambda_3d_3)$ and hence $W_1=span(\diag(1, 0, 0))$, $W_2=span(\diag(0, 1, 0))$, $W_3=span((0,0, 1))$ are invariant subspaces of $\p'$. Hence $\sieg^\Ga$ is of type A III(1,1) $\times$  A III(1,1) $\times$  A III(1,1).\ \\

We postpone the analysis of the uniformizing symmetric space $\sieg^\Ga$ of the other families of non-cyclic coverings of the line, namely of families \textbf{(26, 31, 32)} to the next section, where we present the study of these varieties as families of Galois covers of elliptic curves. %In fact, called  (1e), (2e) (3e), (4e), (5e), (6e) the $6$ families obtained as Galois covers of elliptic curves, it was shown in \cite{fpp} that:
%\begin{itemize}
%	\item (1e) gives the same subvariety as (26) of Table 2 in \cite{fgp}.
%	\item (3e) gives the same subvariety as (31) of Table 2 in \cite{fgp}.
%	\item (4e) gives the same subvariety as (32) of Table 2 in \cite{fgp}.
%	\item (5e) gives the same subvariety as (34)=(23)=(7) of Table 2 in \cite{fgp}.
%\end{itemize}
% Since $\dim \k'=2$, we conclude that $\sieg^\Ga$ is of type BD I $(p=2,q=2)$.
%Summarizing we have proved the following 
%\begin{teo}
%	The uniformizing symmetric spaces for the examples in Table 2 of %
%	$\geq 2$ are the following:
	
%	\begin{center}
%			\begin{tabular}[H]{ccccc}
%				
%				& $g$  & G  & $\dim_{\mathbb{C}}$ & $\sieg^{\Ga}$\\
%				\midrule
%				$(2)$&  2 & $\mathbb{Z}/2 $ & 3 & $\sieg_3$\\
%				$(6)$ & 3 & $\mathbb{Z}/3 $ & 2 & $B_2(\mathbb{C})$ \\
%				$(8)$ & 3 & $\mathbb{Z}/4$ & 2&  $B_2(\mathbb{C})$  \\
%				(10) & 4 & $\mathbb{Z}/3$ & 3 & $B_3(\mathbb{C})$ \\
%				(14) & 4 & $\mathbb{Z}/6$  & 2 & $B_2(\mathbb{C})$ \\
%				(16) & 6 & $\mathbb{Z}/5$ & 2 & $B_2(\mathbb{C})$ \\
%				(26) & 2 & $\mathbb{Z}/2\times \mathbb{Z}/2$ & 2 &  $B_1(\mathbb{C})\times B_1(\mathbb{C})$ \\
%				(27) & 3 & $\mathbb{Z}/2\times \mathbb{Z}/2$ & 3 & $B_1(\mathbb{C})\times B_1(\mathbb{C})\times B_1(\mathbb{C})$  \\
%				(31) & 3& $S_3$ & 2 & $B_1(\mathbb{C})\times B_1(\mathbb{C})$ \\
%				(32) & 3 & $D_4$  & 2 &  $B_1(\mathbb{C})\times B_1(\mathbb{C})$ \\
%				\bottomrule
%			\end{tabular}
%	\end{center}
%\end{teo} 
\subsection{Galois coverings of elliptic curves.}

Let $f_t:C_t\rightarrow C_t'$, $t\in B$, be a family of Galois covers with group $G$ and $g'\geq 1$. It is a result of \cite{fgs}, that $f_t:C_t\rightarrow C_t'$ gives raise to a Shimura variety only for $g'=1$, and that, in this case, it is one of the families (1e), (2e), (3e), (4e), (6e) found in \cite{fpp}. It is proved in \cite[Corollary 4.1]{fgs} that families (1e), (2e), (3e), (4e), (6e) are fibered in totally geodesic curves via their generalized Prym maps 
\begin{align*}
P:B \rightarrow \A^{\delta}_{g-1}, & \quad  [C_t\rightarrow C_t'] \mapsto   Prym(C_t,C_t)
\end{align*}
%$$P:B \rightarrow \A_{g-1},\quad [C_t\rightarrow C_t'] \mapsto Prym(C_t,C_t)$$
%$$\phi: B\rightarrow \A_1, \quad [C_t\rightarrow C_t'] \mapsto [JC_t'].$$
(the Prym map of family (5e) is constant) and are fibered in totally geodesic subvarieties of codimension 1 via the map \begin{align*}
\phi: B\rightarrow \A_1, & \quad [C_t\rightarrow C_t'] \mapsto [JC_t'].
\end{align*} %Therefore they contain infinitely many totally geodesic subvarieties.
 Moreover they prove that countably many of these totally geodesic fibers are Shimura. 

\ \\
Linked to the study of these maps is the decomposition, up to isogeny, of the Jacobian $JC$ of $C$, as $JC \sim JC' \times Prym(C,C')$.
Aim of this section is to study this decomposition at the level of the Siegel space, relating it to the study of the uniformizing symmetric space of the examples. Let $$\pi: \sieg \rightarrow \A_g$$ be the natural projection map and $$B\overset{h}{\rightarrow} \M_g\overset{j}{\rightarrow} \A_g$$ be, respectively, the natural map associated to the family, and the Torelli map. Recall that $h$ is generically finite. Our result is the following:

\begin{teo}\label{teo2}
	%The uniformazing symmetric space $\sieg^\Ga$ relative to families (1e), (2e) (3e), (4e), (6e), admit a decomposition as $\Lambda \times M$, where $\Lambda$ is of type AIII(1,1) and $M$ is irreducible of codimension $1$.
	%Then the image in $\A_g$ of the factor of type A III (1,1) of $\sieg^\Ga$, which is irreducible of dimension $1$, is an irreducible component of the fiber of the Prym map, whereas the image in $\A_g$ of the factor $M$ of $\sieg^\Ga$, which is irreducible of codimension $1$, is an irreducible component of the of the fiber of $\phi$.
	
	Let $\sieg^\Ga$ be the uniformizing symmetric space associated to one of the families (1e), (2e) (3e), (4e), (6e). Then\\
	\begin{enumerate}
		\item [i)] $\sieg^\Ga$ decomposes as $B_1(\mathbb{C})\times M$, where $M$ is an hermitian symmetric space of codimension $1$.\\
		\item [ii)] ${\pi(M)}= \overline{(j\circ h) (F)}$, where $F$ is an irreducible component of the fiber of $\phi$. In particular, $M$ uniformizes $\overline{(j\circ h) (F)}$.\\
		\item [iii)] ${\pi(B_1(\mathbb{C}))}= \overline{(j\circ h) (F)}$, where $F$ is an irreducible component of the fiber of the Prym map. \\
	\end{enumerate}
\end{teo}
\begin{remark}
		It is known that the fibers of the Prym map are not irreducible for the family (1e), and irreducible for the family (2e) (see \cite{fns} and references therein). %Thus, in these cases, the statement of Theorem \ref{teo2} can be made more precise as follows: the fibers of the Prym map of family (2e) are irreducible and $\pi(B_1(\mathbb{C}))= \overline{(j\circ h) (F)}$, where $F$ is the fiber of the Prym map. The fibers of the Prym map of family (1e) are not irreducible and $\pi(B_1(\mathbb{C}))= \overline{(j\circ h) (F)}$, where $F$ is an irreducible component of the non-irreducible fiber of $P$. 
\end{remark}
Consider as usual the action of $G$ on the space of holomorphic one$-$forms and let $H^0(C, K_C)=\oplus_{\chi}V_{\chi}$ be the associated decomposition in isotypic components. If $\chi_0$ is the character of the trivial representation, then $V_{\chi_0}=H^0(C,K_C)^G$. Denote by $V_-$ the direct sum of the other isotypic components. We get $H^0(C, K_C)=H^0(C, K_C)^G\oplus V_-$. We point out that $H^0(C, K_C)^G$ and $ V_-$ are the subspaces of $H^0(C, K_C)$ corresponding, respectively, to the construction of $JC'$ and of $Prym(C,C')$. In fact, since the pullback map $f^*:H^0(C', K_{C'})\rightarrow H^0(C, K_C)$ is injective,  $H^0(C', K_{C'})\simeq f^*(H^0(C', K_{C'}))=H^0(C, K_C)^G$. Note that this implies $\dim H^0(C, K_C)^G=1$. Moreover recall that, for a certain sublattice $\Lambda\subset H_1(C, \mathbb{Z})$, we have $Prym(C, C')=(V_-)^*/\Lambda$. \\

For $J\in\sieg^\Ga$ denote as usual by $\g=\mathfrak{sp}(2g, \erre)$, and $\g=\k\oplus \p$ the Cartan decomposition associated to $\sieg$. The proof of Theorem \ref{teo2} is based on the following two Propositions.

\begin{prop} Let $\g'=\p'\oplus \k'$ be the Cartan decomposition of the uniformizing symmetric space $\sieg^\Ga$ of one of the families (1e), (2e) (3e), (4e), (6e). Then 
$$\p'=W_1\oplus W_2, \quad W_1=S^2(H^0(K_C)^G)^*,\quad W_2\subset S^2(V_-)^*$$
where $W_i$ is $\ad_{\k'}-$invariant. In other words, $\sieg^\Ga=B_1(\mathbb{C})\times M$, where $M$ is a symmetric space with $T_JM=W_2$ and $T_JB_1(\C)=W_1$.
\end{prop}
\begin{proof}
From the decomposition $H^0(K_C)=H^0(K_C)^G\oplus V_-$, we get $$(S^2H^0(K_C))^G=S^2H^0(K_C)^G\oplus (S^2V_-)^G.$$ 
Recalling that $\p\simeq S^2H^0(K_C)^*$, we conclude \begin{equation}\label{Ginv}
\p'=W_1\oplus W_2, \quad W_1=S^2(H^0(K_C)^G)^*,\  W_2= (S^2(V_-)^*)^G\subset S^2(V_-)^*.
\end{equation}
We want now to show that the subspaces $W_i$ are $ad_{\k'}-$invariant and hence that we get a decomposition of $\sieg^\Ga$ as symmetric space. % The action $\rho$ of $G$ on $H^0(C, K_C)$ splits as $$\rho: G \rightarrow GL(H^0(K_C)^G\oplus V_-), \quad \rho(g)=\left( \begin{matrix}
%1 & 0 \\ 0 & \rho_{V_-}(g)
%\end{matrix}
%\right).$$
Choose a norm one vector that spans $H^0(K_C)^G$ a unitary basis of $V_-$, and consider the induced basis of $H^1(C, \C)$.
With the choice of this basis, \eqref{Ginv} implies that, in terms of the usual matrix notation \ref{dec}, an element $D\in \p'=\Z_\p(\Ga)$ is a block matrix of the form $$D=\left(\begin{matrix}
d_{11} & 0 \\ 0 & F
\end{matrix}\right).$$
Since $\k'=[\p',\p']$, the elements in $\k'$ also inherit this property. Hence $$W_1=\{\left(\begin{matrix}
d_{11} & 0 \\ 0 & 0
\end{matrix}\right),\  d_{11}\in \mathbb{C}\}\quad \text{and}\quad W_2=\{\left(\begin{matrix}
0 & 0 \\ 0 & F
\end{matrix}\right)\in \p'\}$$ are $\ad_{\k'}-$invariant subspaces of $\p'$. That is, the adjoint representation is not irreducible and $\sieg^\Ga\simeq B_1(\mathbb{C})\times M$, where $M$ is a symmetric space with tangent space $W_2$, and $T_JB_1(\C)=W_1$.
\end{proof}

\begin{lemma}\label{totallygeo}
	Let $M_1, M_2\subset \sieg$ be closed connected complex submanifolds such that $\pi(M_1)\subset \pi(M_2)$. Then there exists $a\in Sp(2g, \mathbb{Z})$ such that $a.(M_1)\subset M_2$.
\end{lemma}
\begin{proof}
For $a\in Sp(2g, \mathbb{Z})$, let $f_a: \sieg \rightarrow \sieg$, $f_a(J)=a.J=aJa^{-1}$. By assumption $$M_1\subset \bigcup_{a\in Sp(2g, \mathbb{Z})} f_a^{-1}(M_2).$$
Indeed, for $J\in M_1$ there exists $b,c\in Sp(2g, \mathbb{Z})$ and $J_2\in M_2$, such that $bJ_1b^{-1}=cJ_2c^{-1}.$ Hence there exists $a\in Sp(2g, \mathbb{Z})$ such that $f_a(J_1)=J_2$. Thus %For each $a\in Sp(2g, \mathbb{Z})$, the preimage $f_a^{-1}(M_2)\subset \sieg$ is an analytic subset, so we can find compact subsets $\{K_{a,i}\}_{i\in \mathbb{N}}$ of $\sieg$ such that
$$M_1= \bigcup_{a\in Sp(2g, \mathbb{Z})}M_1\cap f_a^{-1}(M_2).$$
Since $M_1\cap f_a^{-1}(M_2)$ is closed in $M_1$ and $a$ vary in a countable set, Baire theorem \cite[p. 57]{bredon} implies that there exists $a$ and an open subset $U\subset M_1$, such that $U\subset f_a^{-1}(M_2)$, i.e. $f_a(U)\subset M_2$. Therefore $f_a^{-1}(M_2)\cap M_1$ is an analytic subset of $M_1$, which contains an open subset $U\subset M_1$. By the Identity Lemma \cite[p. 167]{grauert-remmert-cas}, this implies that $f_a^{-1}(M_2)\cap M_1=M_1$ and hence that $f_a(M_1)\subset M_2$.
\end{proof}

\begin{prop}\label{blabla}
Let $\sieg^\Ga=B_1(\mathbb{C})\times M$ be the uniformizing symmetric space of one of the families (1e), (2e) (3e), (4e), (6e). Then 
\begin{enumerate}
	\item [i)] ${\pi(M)}= \overline{(j\circ h) (F)}$, where $F$ is an irreducible component of the fiber of $\phi$.\\
	\item [ii)] ${\pi(B_1(\mathbb{C}))}= \overline{(j\circ h) (F)}$, where $F$ is an irreducible component of the fiber of the Prym map.
\end{enumerate}
\end{prop}
\begin{proof} \emph{i)} For $t\in B$ generic, $dh_t(T_tB)=H^0(C_t, T_{C_t})^G$, and the codifferential of $\phi$ at $t$ is given by the composition
$$S^2(H^0(K_C)^G)\overset{m \circ i}{\longrightarrow} H^0(2K_C)^G\overset{dh^*}{\longrightarrow} (T_tB)^*, $$
where $i: S^2H^0(K_C)^G\hookrightarrow S^2H^0(K_C)$ is the inclusion and $m$ is the multiplication map. Since $m^*=dj$, we get $d\phi=i^*\circ dj \circ dh$ and we conclude that $$d\phi\circ d(j\circ h)^{-1}=i^*: (S^2H^0(K_C)^*)^G\rightarrow S^2(H^0(K_C)^G)^*, \quad \lambda\mapsto \lambda|_{S^2H^0(K_C)}.$$ Since $W_2\subset S^2(V_-)^*$, it follows that $W_2\subset \ker(d\phi \circ d(j\circ h)^{-1})=d(j\circ h)(\ker d\phi)=d(j\circ h)(T_tF).$ So ${\pi(M)}\subset\overline{(j\circ h)(F)}$. By \cite[Theorem 3.11]{fgs}, $\overline{(j\circ h)(F)}\subset \A_g$ is totally geodesic of dimension $\dim\sieg^\Ga-1$. Denote by $M'$ its uniformizing symmetric space. Applying Lemma \ref{totallygeo} to $M$ and $M'$, we obtain that $a.M\subset M'$ for some $a\in Sp(2g, \mathbb{Z})$. Since they have the same dimension, the equality sign holds and we get $\pi(M)=\pi(a.M)=\pi(M')=\overline{(j\circ h)(F)}$.\\ \\
 %The statement follows by recalling that a fiber of $\phi$ has codimension $1$ in $B$.\\
\emph{ii)} For $t\in B$ generic the codifferential of the Prym mat at $t$ is given by the composition
$$S^2(V_-)\overset{pr}{\longrightarrow} (S^2(V_-))^G \overset{m\circ i}{\longrightarrow}H^0(2K_{C_t})^G\overset{dh^*}{\longrightarrow}(T_tB)^*$$
where $pr$ denotes the natural projection and now $i: (S^2V_-)^G\hookrightarrow S^2H^0(K_C)$. Thus $dP=pr^*\circ i^* \circ dj \circ dh$ and we conclude that 
$$dP\circ d(j\circ h)^{-1}=pr^* \circ i^*: (S^2H^0(K_C)^*)^G\rightarrow S^2(V_-)^*.$$
Since $i^*:S^2H^0(K_C)^* \rightarrow {(S^2V_-)^G}^*$ is the restriction and $W_1=S^2H^0(K_C)^G$, we get $W_1\subset \ker(dP \circ d(j\circ h)^{-1})=d(j\circ h)(\ker dP)=d(j\circ h)(T_tF).$ So ${\pi(B_1(\C))}\subset\overline{(j\circ h)(F)}$. %Since a fiber of the Prym map has dimension $1$, this proves the statement.By Applying Lemma \ref{totallygeo}
By \cite[Theorem 3.9]{fgs}, $\overline{(j\circ h)(F)}\subset \A_g$ is totally geodesic of dimension $1$. Applying again Lemma \ref{totallygeo}, the same argument as above implies ${\pi(B_1(\C))}=\overline{(j\circ h)(F)}$.
\end{proof}

We present in the following the analysis of the various examples, reporting the isomorphisms with the families of Galois covers of the line. For $G=\mathbb{Z}/m$, we will denote by $\chi_j$ character of the the irreducible representation $\rho_j$ of $G$ defined as $$\rho_j(1)(\omega)=\zeta^j\omega,\quad \forall \omega\in H^0(C, K_C),$$  where $\zeta=e^{2\pi i/m}$.\ \\
 
\paragraph{\textbf{Family (1e)=(26)}}
The Chevalley-Weil formula gives $\chi_{\rho}=\chi_0+ \chi_1 $. %	\begin{tabular}[H]{c|cc}
%		
%		& $0$ & $1$  \\
%		\midrule
%		$\chi_1$ & 1 & 1  \\
%		$\chi_2$ & 1 & -1  \\
%		\bottomrule
%	\end{tabular}
%\end{center}
%\ \\\\   
It follows the decomposition of $H^0(C, K_C)$ as $H^0(C, K_C)=H^0(C, K_C)^G\oplus V_-$, with $H^0(C, K_C)^G=V_0$ and $V_-=V_1$, both of dimension $1$. Thus the action $G$ on $H^0(C, K_C)$ is given by  $$\rho: \mathbb{Z}/2\mathbb{Z} \rightarrow GL(V_0\oplus V_1), \quad \rho(g)=\left( \begin{matrix}
\rho_0(g) & 0 \\ 0 & \rho_1(g)
\end{matrix}
\right).$$
Thus, $A_0=\diag(1,-1)$.\\
%$U=\left( \begin{matrix}
%C & \overline{D} \\ D & \overline{C}
%\end{matrix}
%\right)\in \Z_{\g}(A)$ iff
%$ C=\diag(i\lambda_1, i\lambda_2),$ and $ D=\diag(d_1,d_2),$ with $ \lambda_i\in\erre,\  d_i\in \mathbb{C}.$\\
$\p'= \Z_{\p}(A)=\{\left( \begin{matrix}
0 & \overline{D} \\ D & 0
\end{matrix}
\right),\ D=\diag(d_1,d_2),\ d_i\in \mathbb{C}\}$\\
$\k'=[\p', \p']=\Z_{\g}(A)\cap \k=\{\left( \begin{matrix}
C & 0 \\ 0 & \overline{C}
\end{matrix}
\right), \ C=\diag(i\lambda_1, i\lambda_2),\ \lambda_i\in\erre\}.$\\
$W_1=span(\diag(1, 0))= S^2(H^0(C, K_C)^G)^*$,\\ $W_2=span(\diag(0, 1))=S^2(V_-)^*$.\\ $\sieg^\Ga$ decomposes in the product of two irreducible hermitian symmetric spaces of complex dimension $1$.\ \\

\paragraph{\textbf{Family (2e)}}
The Chevalley-Weil formula implies $\chi=\chi_0+2\chi_1$. Thus there exists $1-$dimensional subspaces $V_0, V_1, V_1'\subset H^0(C, K_C)$, such that $H^0(C, K_C)=V_0\oplus V_1\oplus V'_1$ and $$\rho: \mathbb{Z}/2\rightarrow GL(V_0\oplus V_1\oplus V'_1), \quad \rho(g)=\left( \begin{matrix}
\rho_0(g) & 0 & 0\\ 0 & \rho_1(g)&0 \\0 & 0 & \rho_1(g)
\end{matrix}
\right).$$
$A_0=\diag(1,-1,-1)$.\\
%$U=\left( \begin{matrix}
%C & \overline{D} \\ D & \overline{C}
%\end{matrix}
%\right)\in \Z_{\g}(A)$ iff $C=\left( \begin{matrix}
%i\lambda & 0\\ 0 & E
%\end{matrix}
%\right),$ and  $D=\left( \begin{matrix}
%d & 0 \\ 0 & F
%\end{matrix}\right),$ with $ \lambda\in \erre,\ E\in \u(2),\  d\in \mathbb{C},\ F=F^t.$\\
$\p'= \Z_{\p}(A)=\{\left( \begin{matrix}
0 & \overline{D} \\ D & 0
\end{matrix}
\right),\ D=\left( \begin{matrix}
d & 0 \\ 0 & F
\end{matrix}\right),\ d\in \mathbb{C},\ F=F^t\}$.\\ 
$\k'=\Z_{\g}(A)\cap \k=\{\left(\begin{matrix}
C & 0 \\ 0 & \overline{C}
\end{matrix}\right), \ C=\left( \begin{matrix}
i\lambda & 0\\ 0 & E
\end{matrix}
\right),\ \lambda\in \erre,\ E\in \u(2)\}
.$\\
For $C=\left( \begin{matrix}
i\lambda & 0\\ 0 & E
\end{matrix}
\right)\in \k'$ and $D=\left( \begin{matrix}
d & 0 \\ 0 & F
\end{matrix}\right)\in \p'$, we have $$\ad_C(D)=\overline{C}D-DC=\left(\begin{matrix}
-2\lambda d & 0 \\ 0 & \overline{E}F-FE
\end{matrix}\right).$$
$W_1=span(\diag(1,0,0))=S^2(H^0(C, K_C)^G)^*$,\\
$W_2=\{\left(\begin{matrix}
0 & 0 \\ 0 & F
\end{matrix}\right),\ F=F^t\}=S^2(V_-)^*$. \\
$\sieg^\Ga\simeq B_1(\C)\times M$, where $M$ is a $3-$dimensional irreducible symmetric space with Cartan decomposition $\g''=\k''\oplus \p''$ with $\k''=\u(2)$ and $\p''=\{F\in \mathfrak{gl}(2, \mathbb{C}),\ F=F^t\}$. We conclude that $M$ is of type BD I (3,2) and $\sieg^\Ga\simeq$ A III (1,1)$ \times $ BD I (3,2).\ \\

\paragraph{\textbf{Family (3e)=(31)}}
The Chevalley-Weil formula gives $\chi={\chi_0}+{\chi_1}+{\chi_2}$. %where \begin{center}
%	\begin{tabular}[H]{c|ccc}
%		
%		& $0$ & $1$ & $2 $ \\
%		$\chi_1$ & 1 & 1 &1 \\
%		$\chi_2$ & 1 & $\zeta$ & $\zeta^2$ \\
%		$\chi_3$ & 1 & $\zeta^2$ & $\zeta$  \\
%		\bottomrule
%	\end{tabular}
%\end{center}
%with $\zeta=e^{2\pi i /3}$.\\ 
Thus there exists $V_i\subset H^0(C, K_C)$, $i=0,1,2$, $\dim V_i=1$, such that $H^0(C, K_C)=V_0\oplus V_1\oplus V_2$ and $$\rho: \mathbb{Z}/3\mathbb{Z}\rightarrow GL(V_0\oplus V_1\oplus V_2), \quad \rho(1)=\left( \begin{matrix}
\rho_0(1) & 0 & 0\\ 0 & \rho_1(1) & 0 \\ 0 & 0 & \rho_2(1)
\end{matrix}
\right).$$
To get the decomposition $H^0(C, K_C)=H^0(C, K_C)^G\oplus V_-$, recall that  $V_0=H^0(C,K_C)^G$, and $V_-=V_1\oplus V_2$. Denoted by $\zeta=e^{2\pi i /3}$, we have $A_0=\diag(1, \zeta, \zeta^2)$.\\
 %$U\in\g$ is of the form $U=\left( \begin{matrix}
%C & \overline{D} \\ D & \overline{C}
%\end{matrix}
%\right)$, with $D=D^t$ and $C\in \mathfrak{u}(2)$, and $U\in \Z_{\g}(A)$ satisfies
%$$ C=\left( \begin{matrix}
%i\lambda_1 & 0 & 0 \\ 0 & i\lambda_2 & 0 \\ 0 & 0 & i\lambda_2
%\end{matrix}
%\right),\quad D=\left( \begin{matrix}
%d_1 & 0& 0 \\ 0 & 0 & d_2\\ 0 & d_2 & 0
%\end{matrix}\right),\quad  \lambda_i\in\erre,\  d_i\in \mathbb{C}.$$
$\p'= \Z_{\p}(A)=\{\left( \begin{matrix}
0 & \overline{D} \\ D & 0
\end{matrix}
\right),\ D=\left( \begin{matrix}
d_1 & 0& 0 \\ 0 & 0 & d_2\\ 0 & d_2 & 0
\end{matrix}\right),\ d_i\in\mathbb{C}\}$\\
$\k'=\Z_{\g}(A)\cap \k=\{\left( \begin{matrix}
C & 0 \\ 0 & \overline{C}
\end{matrix}
\right), \ C=\diag(i\lambda_1, i\lambda_2, i\lambda_3),\  \lambda_i\in\erre\},\  \lambda_i\in\erre\}.$\\
In this case $\ad_C(D)=\overline{C}D-DC=\left(\begin{matrix}
-2i\lambda_1d_1 & 0 & 0 \\ 0 & 0 & -2i\lambda_2d_2 \\ 0 & -2i\lambda_2d_2 & 0
\end{matrix}\right)$. \\
$W_1=span(\diag(1, 0, 0))=S^2(H^0(C, K_C)^G)^*$,\\
$W_2=span(\left(\begin{matrix}
0 & 0 & 0 \\ 0 & 0 & 1 \\ 0 & 1 & 0
\end{matrix}\right))\subset S^2(V_-)^*$\\
$\sieg^\Ga$ is product of two irreducible hermitian symmetric spaces of complex dimension $1$. \ \\

\paragraph{\textbf{Family (4e)=(32)}}
%Label the characters of $\mathbb{Z}/4\mathbb{Z}$ ad follows:
%\begin{center}
%		& $0$ & $1$ & $2 $ & 3 \\
%		\midrule
%		$\chi_1$ & 1 & 1 &1 & 1\\
%		$\chi_2$ & 1 & -1 & 1 & -1 \\
%		$\chi_3$ & 1 & i & -1& -i  \\
%		$\chi_4$ & 1 & -i & -1 & i \\
%		\bottomrule
%	\end{tabular}
%\end{center}
%\ \\ \\
The Chevalley-Weil formula gives $\chi_{\rho}=\chi_0+ \chi_1 + \chi_3$. 
Thus there exists $V_i\subset H^0(C, K_C)$, $i=0,1,3$, $\dim V_i=1$, such that $H^0(C, K_C)=V_0\oplus V_1\oplus V_3$ and $$\rho: \mathbb{Z}/4\mathbb{Z}\rightarrow GL(V_0\oplus V_1\oplus V_3), \quad \rho(1)=\left( \begin{matrix}
\rho_0(1) & 0 & 0\\ 0 & \rho_1(1) & 0 \\ 0 & 0 & \rho_3(1)
\end{matrix}
\right).$$
$A_0=\diag(1,i,-i).$\\
$\p'= \Z_{\p}(A)=\{\left( \begin{matrix}
0 & \overline{D} \\ D & 0
\end{matrix}
\right),\ D=\left( \begin{matrix}
d_1 & 0& 0 \\ 0 & 0 & d_2\\ 0 & d_2 & 0
\end{matrix}\right),\ d_i\in\mathbb{C}\}$ \\
$\k'=[\p',\p']=\{\left( \begin{matrix}
C & 0 \\ 0 & \overline{C}
\end{matrix}
\right), \ C=\diag(i\lambda_1, i\lambda_2, i\lambda_3),\  \lambda_i\in\erre\},\  \lambda_i\in\erre\}.$\\
The adjoint representation is given by $$\ad_C(D)=\overline{C}D-DC=\left(\begin{matrix}
-2i\lambda_1d_1 & 0 & 0 \\ 0 & 0 & -2i\lambda_2d_2 \\ 0 & -2i\lambda_2d_2 & 0
\end{matrix}\right).$$
$W_1=span(\diag(1, 0, 0))=S^2(H^0(C, K_C)^G)^*$,\\ $W_2=span(\left(\begin{matrix}
0 & 0 & 0 \\ 0 & 0 & 1 \\ 0 & 1 & 0
\end{matrix}\right))\subset S^2(V_-)^*.$ \\
$\sieg^\Ga$ is product of two irreducible hermitian symmetric spaces of complex dimension $1$. \ \\

\paragraph{\textbf{Family (6e)}}
This family was originally studied in \cite{pirola-Xiao}. The Chevalley-Weil formula gives $\chi_{\rho}=\chi_0+ \chi_1 + 2\chi_2$.\\% where \begin{center}
%	\begin{tabular}[H]{c|ccc}
%		
%		& $0$ & $1$ & $2 $ \\
%		\midrule
%		$\chi_1$ & 1 & 1 &1 \\
%		$\chi_2$ & 1 & $\zeta$ & $\zeta^2$ \\
%		$\chi_3$ & 1 & $\zeta^2$ & $\zeta$  \\
%		\bottomrule
%	\end{tabular}
%\end{center}
%with $\zeta=e^{2\pi i /3}$.\\
$A_0=\diag(1, \zeta^2, \zeta^2, \zeta)$, with $\zeta=e^{2\pi i /3}$.\\
%$U=\left( \begin{matrix}
%C & \overline{D} \\ D & \overline{C}
%\end{matrix}
%\right)\in \Z_{\g}(A)$ iff $C=\left( \begin{matrix}
%i\lambda_1 & 0 & 0\\ 0 &  E & 0\\
%0 & 0 & i\lambda_2
%\end{matrix}
%\right)$, with  $E\in \u(2),\ \lambda_i\in \erre$ and  $
$\p'=\{\left( \begin{matrix}
	0 & \overline{D} \\ D & 0
\end{matrix}
\right),\ D=\left( \begin{matrix}
d_1 & 0 & 0 \\ 0 & 0 & d\\0 & d^t & 0
\end{matrix}
\right),\ d_1\in \mathbb{C},\ d\in M(2,1,\mathbb{C})\}$.\\
$\k'=[\p', \p']=\{ \left( \begin{matrix}
C & 0 \\ 0 & \overline{C}
\end{matrix}
\right), \ C=\left( \begin{matrix}
i\lambda & 0 & 0\\ 0 & F & 0\\ 0 & 0 &  \tr(F)
\end{matrix}
\right),\ F\in \u(2)\}.$\\ 
If $C\in \k'$ and $D\in \p'$, then $\ad_C(D)=\left(\begin{matrix}
-2i\lambda d_1 & 0 & 0 \\ 0 & 0 & \overline{E}d-tr(E)d \\
0 & (\overline{E}d-tr(E)d)^t & 0 
\end{matrix}\right)$.
$\sieg^\Ga\simeq B_1(\C)\times B_2(\C)$ . \ \\

We sum up the results obtained in the following table.
\begin{center}
	\begin{tabular}[H]{cccccccc}
		
			\textbf{Family}	& $g(C)$ & $g(C')$ & $|G|$ & G &  $\textbf{m}$ & $\dim$ & Type\\
		\midrule
		{(1e)} & 2 & 1 & 2 &   $\mathbb{Z}/2$ & $(2^2)$ & 2 &   $B_1(\C)\times B_1(\C)$ \\
		(2e)&  3 & 1&  2 & $\mathbb{Z}/2$ & $(2^4)$& 4& $B_1(\C)\times \sieg_3$\\
		{(3e)} & 3 & 1 & 3 &   $\mathbb{Z}/3$ & $(2^2)$ & 2 & $B_1(\C)\times B_1(\C)$ \\
		{(4e)} & 3 & 1 & 4 &   $\mathbb{Z}/4$ & $(2^2)$ & 2 & $B_1(\C)\times B_1(\C)$ \\
	
	    {(6e)} & 4 & 1 & 3 &   $\mathbb{Z}/3$ & $(3^3)$ & 3 & $B_1(\C)\times B_2(\C)$  \\
		\bottomrule
	\end{tabular}
\end{center}\ \\


\begin{thebibliography} {99}
	\bibitem{bredon} G.~E. Bredon, \newblock ``Topology and geometry'',
\newblock Springer-Verlag, New York. 1993.
  
\bibitem{coleman} R.~Coleman, \newblock {\em Torsion points on
    curves}, \newblock In ''Galois representations and arithmetic
  algebraic geometry'', 235--247. \newblock Y. Ihara (ed.), \newblock
  North--Holland, Amsterdam, 1987.
  
\bibitem{cfg} E.~Colombo, P.~Frediani, and A.~Ghigi, \newblock {\em On
    totally geodesic submanifolds in the {J}acobian locus}, \newblock
  {Internat. J. Math.}, 26 (2015), no. 1, 1550005.


\bibitem{dejong-zhang} J.~de~Jong and S.-W. Zhang, \newblock {\em
    Generic abelian varieties with real multiplication are not
    {J}acobians}, \newblock In: ''Diophantine geometry'', U. Zannier
  (ed.), volume~4 of {\em CRM Series}, pages 165--172. Ed. Norm.,
  Pisa, 2007.


\bibitem{farkaskra} H.M.~Farkas and I.~Kra \newblock ``Riemann surfaces'', \newblock Springer-Verlag, New York-Heidelberg-Berlin, 1980.


\bibitem{fgp} P.~Frediani, A.~Ghigi and M.~Penegini, \newblock {\em Shimura
  varieties in the Torelli locus via Galois coverings},
    Int. Math. Res. Not. IMRN, 2015, no. 20, 10595-10623.

\bibitem{fpp} P.~Frediani, M.~Penegini and P.~Porru, {\em  Shimura
  varieties in the Torelli locus via Galois coverings of elliptic
  curves},    Geom. Dedicata, 181 (2016) 177-192.


\bibitem{fpi} P.~Frediani  and G.P.~Pirola,  \newblock {\em On the geometry of
  the second fundamental form of the Torelli map},  \newblock
Preprint 2019. \newblock ArXiv: 1907.11407.  
%\newblock To appear on Proc. Amer. Math. Soc.
  

  
\bibitem{fp} P.~Frediani and P.~Porru, {\em On the bielliptic and
    bihyperelliptic loci}, \newblock Preprint 2018.{\em ArXiv:
    1807.02073}, \newblock To appear on Michigan Math. J.


\bibitem{fgs} P.~Frediani, A.~Ghigi and I.~Spelta, \newblock {\em Infinitely many Shimura varieties in the Jacobian locus for $g\leq  4$},
Preprint 2020. ArXiv: 1910.13245. To appear in Ann. Scuola Norm. Sup. Pisa Cl. Sci.

\bibitem{fns} P.~Frediani, J.C.~Naranjo and I.~Spelta, \newblock {\em The fibres of the ramified Prym map},
Preprint 2020. ArXiv: 2007.02068.


%\bibitem{deba} A.~Ghigi,  \newblock {\em On some differential-geometric
%  aspects of the Torelli map},  \newblock  Boll. Unione Mat. Ital.
%  12 (2019), no. 1-2, 133-144.

\bibitem{gpt} A.~Ghigi and G.P.~Pirola and S.~Torelli,   \newblock {\em
  Totally geodesic subvarieties in the moduli space of curves}, 
\newblock Preprint 2019,  ArXiv: 1902.06098. \newblock To appear on Commun. Contemp. Math.

\bibitem{gleissner} C. ~Gleissner, {\em Threefolds Isogenous to a Product and Product quotient Threefolds with Canonical Singularities}, Master
Thesis, Bayreuth, 2016.

\bibitem{grauert-remmert-cas} H.~Grauert and R.~Remmert,  \newblock
''Coherent analytic sheaves'',  \newblock Springer-Verlag,
Berlin, 1984.


\bibitem{gm1} S.~Grushevsky and M.~M\"oller,  \newblock {\em Shimura curves
	within the locus of hyperelliptic {J}acobians in genus 3},  \newblock
Int. Math. Res. Not. IMRN, (6):1603--1639, 2016.

 \bibitem{hain} R.~Hain.  \newblock {\em Locally symmetric families of
     curves and {J}acobians}, \newblock In: ''Moduli of curves and
   abelian varieties'', 91--108. Vieweg, Braunschweig, 1999.

 \bibitem{helgason} S.~Helgason,  \newblock ``Differential geometry,
   {L}ie groups, and symmetric spaces'', \newblock Academic Press
   Inc., New York, 1978.



\bibitem{knvol2} S.~Kobayashi and K.~Nomizu, \newblock ''Foundations of
differential geometry. {V}ol {II}'', \newblock Interscience
Publishers, New York-London-Sydney, 1969.  

	\bibitem{liu-yau-ecc} K.~Liu, X.~Sun, X.~Yang, and S.-T. Yau,
\newblock {\em Curvatures of moduli spaces of curves and applications},
\newblock Asian J. Math. 21 (2017), no. 5, 841-854.

\bibitem{lz} X.~Lu and K.~Zuo,  \newblock {\em The {O}ort conjecture for on
	{S}himura curves in the {T}orelli locus of curves},  \newblock 
J. Math. Pures Appl. (9) 123 (2019), 41-77.

\bibitem{moonen-linearity-1} B.~Moonen,  \newblock {\em Linearity
  properties of {S}himura varieties. {I}},  \newblock J. Algebraic
    Geom. 7(3):539--567, 1998.

\bibitem{moonen-special} B.~Moonen,  \newblock {\em Special subvarieties
  arising from families of cyclic covers of the projective line},
  \newblock  Doc. Math. 15:793--819, 2010.
 
\bibitem{moonen-oort} B.~Moonen and F.~Oort,  \newblock {\em The {T}orelli
  locus and special subvarieties},  \newblock In: ``{H}andbook of
    {M}{oduli: Volume II}'', G. Farkas and I. Morrison (eds.), 549--94.  International {P}ress,
  Boston, 2013.
 
\bibitem{mumford-Shimura} D.~Mumford,  \newblock {\em A note of {S}himura's
  paper ``{D}iscontinuous groups and abelian varieties''},  \newblock
   Math. Ann. 181:345--351, 1969.


\bibitem{oort-can} F.~Oort,  \newblock {\em Canonical liftings and dense
  sets of {CM}-points},  \newblock In: ''Arithmetic geometry'', Sympos. Math. XXXVII, 
  228--234. Cambridge Univ. Press, Cambridge, 1997.

\bibitem{pirola-Xiao} G.~P. Pirola,  \newblock {\em On a conjecture of
  {X}iao},  \newblock J. Reine Angew. Math. 431:75--89, 1992.


\bibitem{rojas} A.~M. Rojas, {\em Group actions on Jacobian
    varieties}, Rev. Mat. Iber. {23} (2007), 397--420.


% \bibitem{serre-echelon} J.~P. Serre, \newblock Appendix to: \newblock
%   Grothendieck, {\em Techniques de construction en g\'eomtrie
%     analytique. X.} \newblock S\'eminaire Henri Cartan, Volume 13
%   (1960-1961) no. 2, Talk no. 17.

\end{thebibliography}
\end{document}